\PassOptionsToPackage{dvipsnames,x11names}{xcolor}

\documentclass[graybox]{svmult}

\usepackage{array}
\usepackage{subfig}
\usepackage{xspace}
\usepackage{enumerate}
\usepackage{tikz}
\usepackage{comment}
\usepackage{stmaryrd}
\usepackage{amsfonts}
\usepackage{graphicx}
\usepackage{epstopdf}
\usepackage{algorithmic}
\usepackage{amsthm}
\usepackage{mathtools}
\usepackage{tcolorbox}
\usepackage{url}
\usepackage[colorlinks=true, linkcolor=OliveGreen, urlcolor=Fuchsia, citecolor=blue]{hyperref}
\usepackage{cleveref}
\usepackage[hyphenbreaks]{breakurl}
\usepackage{pgfplots}

\pgfplotsset{compat=1.18}
\usetikzlibrary{arrows.meta}

\usetikzlibrary{calc, positioning}
\usetikzlibrary{decorations.pathreplacing, calligraphy, arrows, arrows.meta}

\newcommand{\abs}[1]{\left\lvert#1\right\rvert}

\newcommand{\SR}{\text{SR}}

\newcommand{\EE}{\ensuremath{\mathbb{E}}}

\newcommand{\condnum}[1]{\ensuremath{\kappa(#1)}}

\newcolumntype{R}[1]{>{\raggedleft\arraybackslash }b{#1}}
\newcolumntype{L}[1]{>{\raggedright\arraybackslash }b{#1}}
\newcolumntype{C}[1]{>{\centering\arraybackslash }b{#1}}

\DeclareMathOperator{\fl}{fl}

\usepackage{type1cm}        
\usepackage{makeidx}         
\usepackage{graphicx}        
\usepackage{multicol}        
\usepackage[bottom]{footmisc}

\usepackage{newtxtext}       %
\usepackage[varvw]{newtxmath}       

\makeindex             

\begin{document}

\title*{Probabilistic Error Analysis of Limited-Precision Stochastic Rounding: Horner's Algorithm and Pairwise Summation}
\titlerunning{Probabilistic Error Analysis of Limited-Precision Stochastic Rounding}
\author{El-Mehdi El Arar, Massimiliano Fasi, Silviu-Ioan Filip, and Mantas Mikaitis}
\institute{El-Mehdi El Arar, Sorbonne University, CNRS, LIP6, Paris, France, \email{mehdi.elarar@lip6.fr} \\
  Massimiliano Fasi, University of Leeds, Leeds, UK, \email{m.fasi@leeds.ac.uk}\\
  Silviu-Ioan Filip, Université de Rennes, Inria, CNRS, IRISA, Rennes, France, \email{silviu.filip@inria.fr}\\
Mantas Mikaitis, University of Leeds, Leeds, UK, \email{m.mikaitis@leeds.ac.uk}}
%
%
\maketitle

\abstract{Stochastic rounding (SR) is a probabilistic rounding mode that mitigates errors in large-scale numerical computations, especially when prone to stagnation effects. 
Beyond numerical analysis, SR has shown significant benefits in practical applications such as deep learning and climate modelling.
  The definition of classical SR requires that results of arithmetic operations are known with infinite precision. This is often not possible, and when it is, the resulting hardware implementation can become prohibitively expensive in terms of energy, area, and latency. A more practical alternative is \emph{limited-precision SR}, which only requires that the outputs of arithmetic operations are available in higher, finite, precision.
  We extend previous work on limited-precision SR presented in [El Arar et al., \emph{SIAM~J.~Sci.~Comput.}~47(5) (2025), B1227--B1249], which developed a framework to evaluate the trade-off between accuracy and hardware resource cost in SR implementations.
  Within this framework, we study the Horner algorithm and pairwise summation, providing both theoretical insights and practical experiments in these settings when using limited-precision SR.
}

\section{Introduction}
\label{sec:intro}

Finite-precision arithmetic~\cite{mbdj18} is an inherent limitation of computer systems.
In computations comprising long sequences of floating-point operations, it naturally leads to the accumulation of rounding errors, which can degrade numerical accuracy.
Stochastic rounding~\cite{cfhm22,effm26} (SR) addresses this issue by replacing deterministic rounding rules with a probabilistic mechanism that ensures that rounding errors have mean zero.

As a result, SR has gained increasing attention in applications such as iterative solvers, optimization algorithms, and machine learning, where small systematic errors can compound over many arithmetic operations.
However, the unbiased property of SR applies to individual rounding errors and does not necessarily extend to a sequence of operations. For instance, the computation of the variance is biased under SR~\cite{esop23}, for both the textbook and the two-pass algorithms.

Rounding error analysis has been developed to better understand the behavior of algorithms with SR applied to arithmetic operations.
For several algorithms, SR yields probabilistic error bounds that grow as $\mathcal{O}(\sqrt{n}u)$, where $n$ is the problem size and $u$ is the unit roundoff. This is significantly better than deterministic rounding modes, for which worst-case bounds grow as $\mathcal{O}(nu)$. We refer the reader to a survey~\cite{chm21} and a monograph~\cite[Ch.~4~and~5]{thesisarar} on SR for further details.

Implementing SR in hardware requires the generation of random bits to drive the probabilistic rounding decision. El Arar et al.~\cite{effm25} proposed a probabilistic error analysis of SR that only requires a limited number $r$ of random bits. For recursive summation and inner products of length $n$, they suggest setting $r \approx \lceil (\log_2n) / 2\rceil$, which they showed to be an appropriate trade-off point between accuracy and hardware resource cost.

Here, we extend our prior work~\cite{effm25} and study the behaviour of Horner's algorithm for polynomial evaluation and of pairwise summation under limited-precision SR.
The former is used
by \verb+polyval+, \verb+polyvalm+, and \verb+polyfit+ in MATLAB,
by \verb+evalpoly+ in Julia,\footnote{\url{https://docs.julialang.org/en/v1/base/math/\#Base.Math.evalpoly}}
by \verb+numpy.polynomial.polynomial.polyval+ in NumPy,\footnote{\url{https://numpy.org/devdocs/reference/generated/numpy.polynomial.polynomial.polyval.html}}
by polynomial evaluation functions in the GNU Standard Library,\footnote{\url{https://www.gnu.org/software/gsl/doc/html/poly.html}} and
by the \verb+horner+ function in Maple.\footnote{\url{https://www.maplesoft.com/support/help/maple/view.aspx?path=MTM\%2Fhorner}}
Pairwise summation is the default summation algorithm
in Julia,\footnote{\url{https://github.com/JuliaLang/julia/blob/ce9da6b2dc775f2bf201be0b8371db32a2458741/base/reduce.jl\#L534}} and
in the Apache Arrow Library.\footnote{\url{https://github.com/apache/arrow/blob/d08d5e64fcfd8759d3a7089eced3e9a2d7a17f20/cpp/src/arrow/compute/kernels/aggregate_internal.h\#L157}}
Under classical SR, one can derive probabilistic error bounds that grow as $\mathcal{O}\bigl(\sqrt{n}u_p\bigr)$ for Horner's algorithm~\cite{esop22} and as $\mathcal{O}\bigl(\sqrt{\log_2(n)}u_p\bigr)$ for pairwise summation~\cite{esop23}.
For limited-precision SR, we establish probabilistic error bounds proportional to $\sqrt{n} u_p + nu_{p+r}$ for Horner's algorithm and to $\sqrt{\log_2 (n)}u_p + \log_2(n)u_{p+r}$ for pairwise summation, where $p$ denotes the working precision and \mbox{$u_k := 2^{1-k}$}.

\section{Stochastic Rounding}
\label{sec:sr}

Let $\mathbb{F}\subset \mathbb{R}$ denote a normal floating-point number system with $p$ digits of precision, and let $x\in \mathbb{R}$. We denote the smallest precision-$p$ floating-point number no smaller than $x$ by $\llceil x \rrceil_p $, and the largest floating-point number no greater than $x$ by $\llfloor x \rrfloor_p$. In other words, we have
$$ \llceil x \rrceil_p=\min\{y\in \mathbb{F} : y \geq x\},  \quad \llfloor x \rrfloor_p=\max\{y\in \mathbb{F} : y \leq x\},$$
and by definition, $\llfloor x \rrfloor_p  \leq x \leq \llceil x \rrceil_p$, with equality throughout if and only if $x \in\mathbb{F}$.
A non zero real number $x \not\in \mathbb{F}$ has two possible rounding candidates in floating-point arithmetic $\llfloor x \rrfloor_p$ or $\llceil x \rrceil_p$, which coincide if $x \in \mathbb{F}$. \emph{Rounding} is an operation $\fl : \mathbb{R} \to \mathbb{F}$ that maps $x$ to either $\llfloor x \rrfloor_p$ or $\llceil x \rrceil_p$, and it can be shown that the rounded quantity satisfies
\begin{equation*}
    \fl(x) =x(1+\delta), \label{fl(x)}
\end{equation*}
where the relative error $\delta = (\fl(x) - x)/x$ is such that $\abs{\delta} < 2^{1-p} = u_p$.

\begin{definition}[Stochastic rounding]
  \label{def:sr}
Let $x \in \mathbb R \setminus \mathbb{F}$. The stochastic rounding of $x$ to precision-$p$ $\mathbb{F}$
is the Bernoulli random variable
\begin{equation}\label{eq:sr}
    \textnormal{SR}_p(x) =
    \begin{cases}
        \llceil x \rrceil_p,   & \text{with probability\ } q(x), \\
        \llfloor x \rrfloor_p, & \text{with probability\ } 1 - q(x),
    \end{cases}
    \qquad
    q(x) = \dfrac{x - \llfloor x \rrfloor_p}{\llceil x \rrceil_p - \llfloor x \rrfloor_p}.
\end{equation}
When $x \in \mathbb{F}$, we have $\SR_p(x) = x$.
\end{definition}

\def\ticksep{0.2}
\def\intsep{0.2}
\def\arrowsep{0.2}
\def\rtwosep{0.2}
\def\compssep{0.3}

\def\barheight{0.4}
\def\barsep{0.5}
\def\barlength{2.8}
\def\hdis{1.9}

\begin{figure}[t]
  \centering
  \begin{tikzpicture}[every node/.style={
      minimum width=0pt,
      minimum height=0pt,
      inner sep=0pt},semithick]


    \node (A) at (0,0) {};
    \node (B) at (3,0) {};
    \node (C) at (3.3,0) {};
    \node (D) at (8,0) {};

    \draw ($(A)-(\compssep,0)$)--($(D)+(\compssep,0)$);

    \foreach \x in {A,B,C,D} 
    \draw ($(\x)+(0,-\ticksep)$) -- ($(\x)+(0,\ticksep)$);


    \node [above=\ticksep+0.1 of A] {$\llfloor x \rrfloor_p$};
    \node [above=\ticksep+0.1 of B] {$x$};
    \node [xshift=0.5cm, above=\ticksep of C] {$\fl_{p+r}(x)$};
    \node [above=\ticksep+0.1 of D] {$\llceil x \rrceil_p$};

    \draw [Stealth-Stealth]($(A)+(0.5*\intsep,-2*\intsep)$) --
    ($(B)+(-0.5*\intsep,-2*\intsep)$);
    \node [below=2.6*\intsep of $(A)!0.5!(B)$] {\small$x - \llfloor x \rrfloor_p$};
    \draw [Stealth-Stealth]($(B)+(0.5*\intsep,-2*\intsep)$) --
    ($(D)+(-0.5*\intsep,-2*\intsep)$);
    \node [below=2.6*\intsep of $(B)!0.5!(D)$] {\small$ \llceil x \rrceil_p - x$};

    \draw [Stealth-Stealth]($(A)+(0.5*\intsep,-6*\intsep)$) --
    ($(C)+(-0.5*\intsep,-6*\intsep)$);
    \node [below=6.6*\intsep of $(A)!0.5!(C)$] {\small$\fl_{p+r}(x) - \llfloor x \rrfloor_p$};
    \draw [Stealth-Stealth]($(C)+(0.5*\intsep,-6*\intsep)$) --
    ($(D)+(-0.5*\intsep,-6*\intsep)$);
    \node [below=6.6*\intsep of $(C)!0.5!(D)$] {\small$\llceil x \rrceil_p - \fl_{p+r}(x)$};


  \end{tikzpicture}
  \caption{Quantities used in the definitions~\eqref{eq:sr} and~\eqref{eq:sr-imp}.}
  \label{fig:sr}
\end{figure}
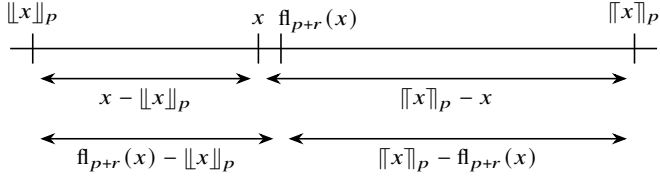

\Cref{fig:sr} depicts the quantities in this definition.
Note that if $x \in \mathbb{F}$, then $q(x) = 0$ and $\SR_p(x) = \llfloor x \rrfloor_p = x$ with probability 1. More generally, for $x \in \mathbb{R}$ we have
\begin{align*}
    \EE(\SR_p(x)) &= q(x)\llceil x \rrceil_p +\big(1-q(x)\big)\llfloor x \rrfloor_p \\
    &=  q(x)(\llceil x \rrceil_p - \llfloor x \rrfloor_p) + \llfloor x \rrfloor_p =x.
\end{align*}

To use Definition~\ref{def:sr}, one must know $x$ with infinite precision, but it is often impossible, or impractical, to compute the exact value of $x$ before rounding.
Therefore, in practice, in $q(x)$ one typically replaces $x$ by $\fl_{p+r}(x)$, a representation of $x$ with $p+r$ bits of precision, for some positive integer $r$. Here $r$ denotes the additional bits of precision available before rounding and, equivalently, the number of random bits needed to perform SR, as we will see later. These quantities are also indicated in Figure~\ref{fig:sr}.

\begin{definition}[limited-precision stochastic rounding]
    Let $x \in \mathbb R \setminus \mathbb{F}$. The limited-precision stochastic rounding of $x$ to precision-$p$ $\mathbb{F}$ using $r$ random bits is the Bernoulli random variable
\begin{equation}
    \label{eq:sr-imp}
    \textnormal{SR}_{p,r}(x) =
    \begin{cases}
        \llceil x \rrceil_p,   & \text{with probability\ } q_r(x), \\
        \llfloor x \rrfloor_p, & \text{with probability\ } 1- q_r(x),
    \end{cases}
    \qquad
    q_r(x) = \frac{\fl_{p+r}(x) - \llfloor x \rrfloor_p}{\llceil x \rrceil_p -  \llfloor x \rrfloor_p}.
\end{equation}
When $x \in \mathbb{F}$, we have $\SR_{p,r}(x) = x$.
\end{definition}

$\SR_{p,r}$ and $\SR_p$ both output $p$-bit precision results. The main difference is that $\SR_{p,r}$ uses $r$ random bits, whereas $\SR_{p}$ essentially assumes an infinite value for $r$. Note that $\fl_{p+r}(x)$ is represented with $p+r$ bits of precision such that
  \begin{equation}
    \label{eq:fl}
    \fl_{p+r}(x) = x(1 + \beta).
  \end{equation}
Similarly to~\cite{effm25}, we take $\fl_{p+r}(x)$ to be the binary representation of $x$ truncated to the first $p+r$ binary digits, for some positive integer $r$.

Unlike $\SR_p$, the limited-precision SR operator is biased, since
\begin{align*}
    \EE\big(\SR_{p,r}(x)\big) &= q_r(x)\llceil x \rrceil_p +\big(1-q_r(x)\big)\llfloor x \rrfloor_p \\
    &=  q_r(x)(\llceil x \rrceil_p - \llfloor x \rrfloor_p) + \llfloor x \rrfloor_p\\
    &= \fl_{p+r}(x),
\end{align*}
and by linearity of the expectation, we have that $\EE\left(\SR_{p,r}(x) -x\right) = \fl_{p+r}(x) -x$.
Let $\delta$ be such that $\SR_{p,r}(x) = x(1+\delta)$. Then, \Cref{eq:fl} yields
\begin{equation}
\label{eq:sr-bias}
    \EE(\delta) = \frac{\fl_{p+r}(x) -x}{x} = \beta.
\end{equation}

\begin{remark}
This formalization of limited-precision SR is already starting to get traction. Indeed, the IEEE P3109 interim report~\cite{ieee25}, which introduces number formats and their arithmetic for Machine Learning, specifies three limited-precision SR variants, offering different tradeoffs between bias and hardware complexity. Called Stochastic[A--C], they differ in how $\fl_{p+r}$ is defined, impacting the value of $\beta$ in~\eqref{eq:fl}. For more details see~\cite{fife25}, \cite[sec.~4.9.3]{ieee25} and~\cite[sec.~3]{effm26}.
\end{remark}

The following result shows that a sequence of errors produced by $\SR_{p,r}$ does not satisfy the mean-independence property. It also provides a decomposition useful to derive tight probabilistic error bounds.

\begin{lemma}[{\cite[lem. 3.10]{effm25}}]
    \label{lem:main-result}
    Let $\delta_1, \delta_2,\ldots, \delta_n$ be random errors produced by a sequence of elementary operations using $\textup{\SR}_{p,r}$, and let $\beta_1, \beta_2,\ldots, \beta_n$ be their corresponding errors incurred by $\fl_{p+r}$. Then, the random variables $\alpha_k = \delta_k - \beta_k$ for $1\leq k \leq n$, are mean independent, that is to say,
    $$
    \EE(\alpha_k \mid \alpha_1,\ldots,\alpha_{k-1}) = \EE(\alpha_k) = 0.
    $$
    Moreover, for all $1\leq i \leq n$,
    \begin{equation*}
      \prod_{k= i}^{n} (1+ \delta_{k}) = \prod_{k= i}^{n} (1+\alpha_{k}) +  \mathcal{B}_i,\qquad
      \mathcal{B}_i = \sum_{\substack{K \in \mathcal{P}(\mathcal I_i)\\ K \neq \mathcal I_i}}
        \left(\prod_{i \in K} (1 + \alpha_{i}) \prod_{j \in \mathcal I_i \setminus K} \beta_{j}\right)
    \end{equation*}
    where $\mathcal I_i = \{k \in \mathbb{N} : i \le k \le n\}$ is such that
        \begin{equation*}
            \abs{\mathcal{B}_i} \leq \gamma_{n-i+1}(u_p +u_{p+r}) -\gamma_{n-i+1}(u_p),\qquad
            \gamma_{m}(x)=(1+x)^{m} -1.
        \end{equation*}
\end{lemma}

\section{Horner's algorithm}
\label{sec:horner}

Polynomial evaluation in floating-point arithmetic is known to be sensitive to rounding errors. In particular, intermediate operations may introduce significant numerical inaccuracies, and in some cases catastrophic cancellations can occur. A widely used method for polynomial evaluation is Horner's algorithm, which provides an efficient way to compute a polynomial's value using a sequence of multiply-add operations. 

Throughout the remainder of the paper, we define $\widehat{x} = \SR_{p,r}(x)$. Consider the polynomial $P(x) = \sum_{i=0}^n a_i x^i$. We will now perform a roundoff error analysis for Horner's algorithm
\begin{equation}
  \label{eq:horner}
P(x)= (((a_nx +a_{n-1})x +a_{n-2})x \ldots +a_1)x +a_0
\end{equation}
using precision-$p$ floating-point arithmetic with $\SR_{p,r}$.
In our analysis, we will rely on the condition number of a polynomial, defined by
\begin{equation}\label{eq:cond-num-horner}
    \condnum{P} = \dfrac{\sum_{i=0}^{n} \abs{a_i} \abs{x}^i}{\abs{\sum_{i=0}^{n} a_i x^i}},
\end{equation}
and on the error function
\begin{equation}\label{eq:gamma_n_u}
\gamma_n(u)= (1+u)^{n}-1 = n u + \mathcal{O}(u^2) \ \text{for} \ nu \ll 1.
\end{equation}

The Horner method (\Cref{eq:horner}) can be applied recursively, yielding
\begin{equation*}
  \begin{array}{rlrl}
    \widehat r_0 &= \fl(a_{n}) = a_n, \qquad  &r_0 &= a_n,\\
    \widehat{r}_{2k-1}&=\fl(\widehat{r}_{2k-2}x) = \widehat{r}_{2k-2} x (1+\delta_{2k-1}), \qquad &  r_{2k-1} &= r_{2k-2}x, \\
    \widehat{r}_{2k} &= \fl(\widehat{r}_{2k-1} +a_{n-k}) = (\widehat r_{2k-1} +a_{n-k})(1+\delta_{2k}), \qquad &  r_{2k} &= r_{2k-1} +a_{n-k}, \\
    \widehat{r}_{2n} &= \widehat{P}(x), \qquad & r_{2n} &= P(x),
  \end{array}
\end{equation*}
where the quantities on the left are computed in finite-precision arithmetic, and those on the right are computed using exact arithmetic.
Note that odd steps correspond to products and even ones to multiplications.

Let $\delta_0 = 0$, and from~\cite[Sec.~5.1]{high02} we have
\begin{equation}
    \label{eq:horner_error}
    \widehat{P}(x) = \sum_{i=0}^{n}a_i x^i \prod_{k=2(n -i)}^{2n} (1+\delta_k).
\end{equation}
Note that for $1 \leq k \leq n$, one has $\abs{\delta_{k}} \leq u_p$
and $\abs{\EE(\delta_{k})}  = \abs{\beta_{k}}  \leq u_{p+r} $,
where the $\beta_{k}$ are defined, analogously to \eqref{eq:sr-bias}, as
\begin{equation*}
\beta_{2k-1} = \frac{\fl_{p+r}(\widehat{r}_{2k-1} +a_{n-k}) -(\widehat{r}_{2k-1}+a_{n-k})}{\widehat{r}_{2k-1} + a_{n-k}} \ \text{and} \ \beta_{2k} = \frac{\fl_{p+r}(\widehat{r}_{2k-2}x) -\widehat{r}_{2k-2}x}{\widehat{r}_{2k-2}x}.
\end{equation*}

Next, we give a probabilistic bound on the relative error of this algorithm.

\begin{theorem}
    \label{thm:proba-horner}
    For any $0 < \lambda < 1$, the quantity $\widehat{P}(x)$ in \eqref{eq:horner_error} satisfies
    \begin{equation}
    \label{eq:proba-error_horner}
     \frac{\abs{\widehat P(x) - P(x)}}{\abs{P(x)}}
     \leq \condnum{P} \left(\sqrt{u_p \gamma_{4n}(u_p)} \sqrt{\ln (2 / \lambda)} + \gamma_{2n}(u_p +u_{p+r}) -\gamma_{2n}(u_p)\right),
    \end{equation}
    with probability at least $1-\lambda$, where $\condnum{P}$ and $\gamma_n(u_{p+r})$ are defined in~\eqref{eq:cond-num-horner} and \eqref{eq:gamma_n_u}, respectively.
\end{theorem}

\begin{proof}
    By \cref{lem:main-result}, the random variables $\alpha_1, \alpha_2,\ldots, \alpha_{n-1}$ with $\alpha_j = \delta_j - \beta_j$ are mean independent and for all $ 0 \le i \le n$,
    $$
    \prod_{k= 2(n -i)}^{2n} (1+\delta_{k}) = 
    \prod_{k= 2(n -i)}^{2n} (1+\alpha_{k}) +  \mathcal{B}_i,
    $$
    with
    \begin{equation}
      \label{eq:Bi}
      \abs{\mathcal{B}_i} \leq \abs{\mathcal{B}_1}
      \leq \gamma_{2n}(u_p +u_{p+r}) -\gamma_{2n}(u_p).
    \end{equation}
      Therefore,
      \begin{equation}
        \label{eq:steps-proof}
        \begin{aligned}
            \abs{\widehat P(x) - P(x)} &= \abs{\sum_{i=0}^{n} a_i x^i \left(\prod_{k= 2(n -i)}^{2n} (1+\delta_{k}) -1 \right)}\\
            &= \abs{\sum_{i=0}^{n} a_i  x^i \left(\prod_{k= 2(n -i)}^{2n} (1+\alpha_{k}) +  \mathcal{B}_i -1\right) }\\
            &\leq \abs{\sum_{i=0}^{n} a_i  x^i \left(\prod_{k= 2(n -i)}^{2n} (1+\alpha_{k}) -1\right)} + \abs{\sum_{i=0}^{n} a_i  x^i \mathcal{B}_i}\\
            &\leq \abs{M} + \sum_{i=0}^{n} \abs{a_i  x^i}\abs{\mathcal{B}_i},
        \end{aligned}
    \end{equation}
    By \cite[Theorem 4.5]{esop23a}, we have
    \begin{equation}
      \label{eq:M}
        \abs{M} \le \sum_{i=0}^{n} \abs{a_i  x^i} \left(\sqrt{u_p \gamma_{4n}(u_p)} \sqrt{\ln (2 / \lambda)}\right)
    \end{equation}
    with probability at least $1-\lambda$. Combining~\eqref{eq:Bi}, \eqref{eq:steps-proof}, and~\eqref{eq:M}, we can conclude that
    $$
    \abs{\widehat P(x) - P(x)} \le \sum_{i=0}^{n} \abs{a_i  x^i} \left(\sqrt{u_p \gamma_{4n}(u_p)} \sqrt{\ln (2 / \lambda)} + \gamma_{2n}(u_p +u_{p+r}) -\gamma_{2n}(u_p)\right)
    $$
    with probability at least $1-\lambda$.
  \end{proof}

  \begin{remark}
    \label{rem:1}
    Interestingly, the probabilistic bound in~\eqref{eq:proba-error_horner} is propotional to $\sqrt{n} u_p + n u_{p+r}$. Therefore, for large $r$ the bound grows as $\sqrt{n} u_{p}$, which is equivalent to the bound for exact SR. Moreover, this bound is similar (asymptotically) to those obtained in the analysis of summation and inner products from~\cite{effm25}. Consequently, the same rule of thumb where neither of the two error terms is dominated by the other, namely $r \approx \lceil (\log_2n) / 2\rceil$, applies here as well.
  \end{remark}

  \begin{remark}
    In practice, Horner's algorithm is often implemented using fused multiply-add (FMA) instructions, which compute expressions of the form $ax+b$ with a single rounding. The present analysis does not account for this implementation detail, which we leave as a subject for future work.
  \end{remark}

\section{Pairwise summation}
\label{sec:pairwise}

Let us now consider the problem of computing the sum
$$
s=\sum_{i=1}^n a_i, \qquad a_{i} \in \mathbb{F}.
$$

Higham~\cite{high93s} showed that computing $s$ in floating-point arithmetic with a binary tree of sums leads to a deterministic error bound that grows as $\mathcal{O}(\log_2(n)u)$. Using different techniques to build the martingale, \cite{esop23, hi23, deps25} proved that using SR leads to a probabilistic error bound that grows as $\mathcal{O}(\sqrt{\log_2(n)}u)$. In this section, we investigate the rounding error of pairwise summation under $\SR_{p,r}$.

Let $h$ be the depth of the summation tree. We can assume without loss of generality that $n=2^h$: if in fact $2^{h-1} < n < 2^h$, setting the remaining $2^h- n$ missing inputs to zero would not alter $s$.
Using the same notation as~\cite{esop23}, in finite precision we have
\begin{eqnarray*}
    \widehat{s} = \sum_{i=1}^{2^h} a_i\prod_{j=1}^h\Big(1+\delta_{\lceil i/2^{j} \rceil}^j\Big).
\end{eqnarray*}

\begin{theorem}
    \label{thm:pairwise}
    For all $0 < \lambda <1$, the computed $\widehat{s}$ satisfies under $\SR_{p,r}$
    \begin{equation}
        \label{eq:proba-error_pairwise}
        \frac{\abs{\widehat{s} - s}}{\abs{s}} \leq  \condnum{a}\left(\sqrt{u\gamma_{2\lceil \log_2(n) \rceil}(u)} \sqrt{
            \ln(2/ \lambda)} + \gamma_{\lceil \log_2(n) \rceil}(u_p +u_{p+r}) -\gamma_{\lceil \log_2(n) \rceil}(u_p)\right),
    \end{equation}
    with probability at least $1-\lambda$.
\end{theorem}

\begin{proof}
    Under $\SR_{p,r}$, the random errors are not mean independent. Like in the proof of Theorem~\ref{thm:proba-horner}, we use~Lemma~\ref{lem:main-result} to separate the martingale and the bias. We have that the random variables $\alpha_1, \alpha_2,\ldots, \alpha_{n-1}$ such that $\alpha_j = \delta_j - \beta_j$ are mean independent and
    $$
    \prod_{j=1}^h \Big(1+\delta_{\lceil i/2^{j} \rceil}^j\Big) =
    \prod_{j=1}^h\Big(1+\alpha_{\lceil i/2^{j} \rceil}^j + \beta_{\lceil i/2^{j} \rceil}^j\Big) =
    \prod_{j=1}^h \Big(1+\alpha_{\lceil i/2^{j} \rceil}^j\Big) +  \mathcal{B}_j,
    $$
    with
    \begin{align*}
        \abs{\mathcal{B}_j} &\leq \abs{\mathcal{B}_1}
        \leq \gamma_{h}(u_p +u_{p+r}) -\gamma_{h}(u_p).
    \end{align*}
    Therefore,
    \begin{equation*}
        \begin{aligned}
            \abs{\widehat s - s} &= \abs{\sum_{i=1}^{2^h} a_i \left(\prod_{j=1}^h \Big(1+\delta_{\lceil i/2^{j} \rceil}^j\Big) -1 \right)}\\
            &= \abs{\sum_{i=1}^{2^h} a_i  \left(\prod_{j= 1}^{h} \Big(1+\alpha_{\lceil i/2^{j} \rceil}^j\Big) +  \mathcal{B}_i -1\right) }\\
            &\leq \abs{\sum_{i=1}^{2^h} a_i \left(\prod_{j= 1}^{h} \Big(1+\alpha_{\lceil i/2^{j} \rceil}^j\Big) -1\right)} + \abs{\sum_{i=1}^{2^h} a_i  \mathcal{B}_i}\\
            &\leq \abs{M} + \sum_{i=1}^{2^h} \abs{a_i}  \abs{\mathcal{B}_i}.
        \end{aligned}
    \end{equation*}
    By \cite[Theorem 3.5]{esop23}, we have
    \begin{equation*}
        \abs{M} \le \sum_{i=1}^{2^h} \abs{a_i} \left(\sqrt{u\gamma_{h}(u)} \sqrt{
            \ln(2/ \lambda)} \right),
    \end{equation*}
    with probability at least $1-\lambda$. Since $n=2^h$, we can conclude that
    $$
    \abs{\widehat s - s} \le \sum_{i=1}^{n} \abs{a_i} U_n,
    $$
    with probability at least $1-\lambda$, where
    $$
    U_n = \sqrt{u_p\gamma_{2\lceil \log_2(n) \rceil}(u_p)} \sqrt{
            \ln(2/ \lambda)} + \gamma_{\lceil \log_2(n) \rceil}(u_p +u_{p+r}) -\gamma_{\lceil \log_2(n) \rceil}(u_p).\qedhere
    $$
\end{proof}

\begin{remark}
  The probabilistic bound in~\eqref{eq:proba-error_pairwise} is propotional to $\sqrt{\log_2(n)} u_p + \log_2(n) u_{p+r}$. For a large $r$, this bound is equivalent to the classical SR bound. Balancing the two terms gives us the corresponding rule of thumb $r \approx \lceil \log_2 (\log_2 (n)) / 2\rceil$. Interestingly, the maximal error accumulated in pairwise summation grows proportionally to $\log_2(n)$, highlighting a key difference with recursive summation, for which the error grows proportionally to $n$. This leads us to posit that the appropriate rule of thumb for algorithms dominated by computation chains with error propagation of length at most $k$ is to consider $r\approx \lceil\log_2(k)/2 \rceil$.
\end{remark}

\section{Numerical experiments}

We present a set of numerical experiments on polynomial evaluation using Horner's algorithm and on summation of floating-point numbers using pairwise summation. We focus on situations that are prone to stagnation when RN is used and investigate the impact of the number of random bits $r$ when $\SR$ is employed. The experiments are performed using the \texttt{srfloat}\footnote{\href{https://github.com/sfilip/srfloat}{https://github.com/sfilip/srfloat}} library, which simulates $\SR_{p,r}$ arithmetic as described in \Cref{sec:sr}.

\subsection{Horner's algorithm}

We evaluate a polynomial with different coefficients using Horner's algorithm, \Cref{eq:horner} at $x=0.9990234375$ using \textit{binary16} arithmetic (\Cref{fig:fp-experiments}), and $x=0.98828125$ in \textit{bfloat16} (\Cref{fig:bf-experiments}), respectively. We compare the results obtained with $\SR_{p,r}$ to those produced by RN. All $\SR_{p,r}$ computations are repeated $30$ times, and we plot the forward error of the average result over the 30 SR instances.

\begin{figure}
    \centering
    \includegraphics[width=0.5\linewidth]{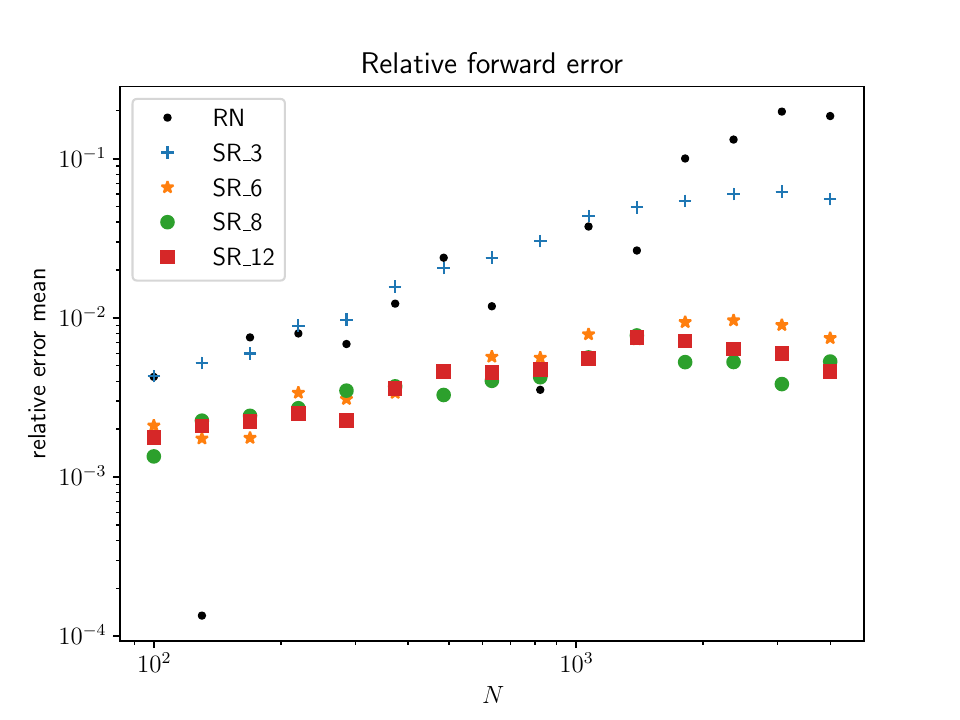}%
    \includegraphics[width=0.5\linewidth]{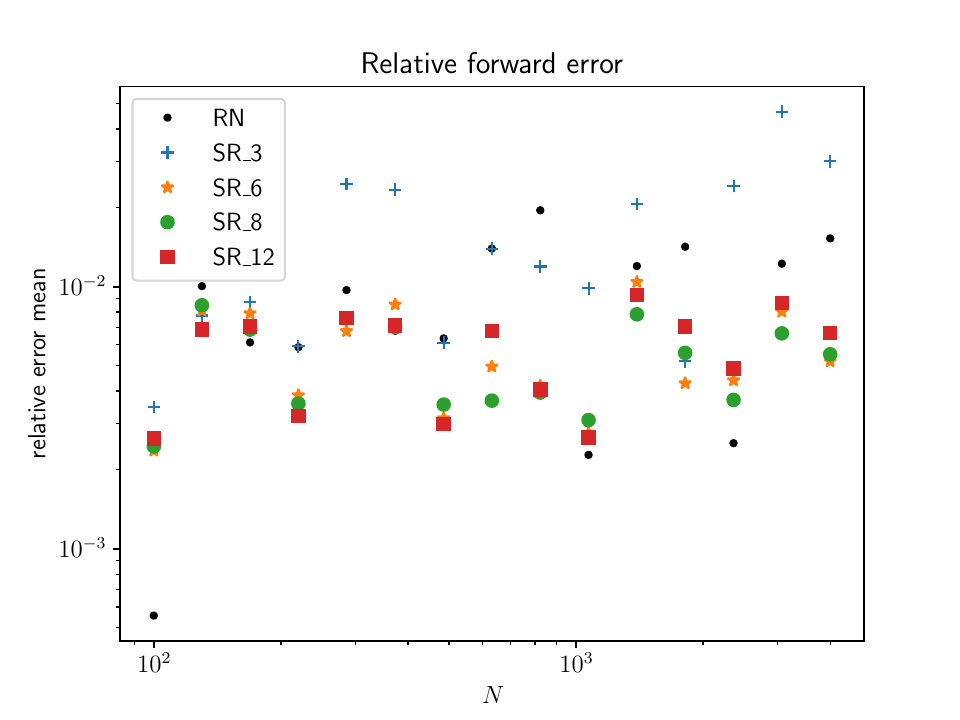}
    \caption{Relative error of RN and $\SR_{11,r}$ in IEEE 754 \textit{binary16} arithmetic ($p=11$). The coefficients are drawn from a uniform distribution over $[0,1]$ (left) and $[-1,1]$ (right).}
    \label{fig:fp-experiments}
\end{figure}

\Cref{fig:fp-experiments} illustrates two different behaviors depending on the distribution of the polynomial coefficients. In the left, where the coefficients are drawn from a uniform distribution over $[0,1]$, the evaluation with \textit{binary16} RN and \textit{binary16} $\SR_{11,r=3}$ clearly exhibits stagnation. In the case of $\SR_{11,r=3}$, the injected noise (with $r=3$) is very small, which slightly mitigates the stagnation effect, but does not eliminate it completely.
In this situation, the relative error grows as the polynomial degree $N$ increases. In contrast, \textit{binary16} $\SR_{11,r}$, for $r=6, 8 \ \text{and} \ 12$, significantly mitigate this effect and yield smaller errors. This behavior confirms the well-known advantage of SR in low precision arithmetic when accumulation of rounding errors leads to stagnation.

In the experiments of the right-hand side of Figure~\ref{fig:fp-experiments}, where the coefficients are uniformly distributed over $[-1,1]$, the relative errors obtained with \textit{binary16} RN and \textit{binary16} $\SR_{11,r}$ are of comparable magnitude, with slight adventage to $\SR_{11,r}$ errors, for $r=6, 8$ and $12$. In this case, the coefficients have mixed signs, which produces rounding errors that are approximately symmetrically distributed. As a consequence, the absorption errors occurring with RN tend to compensate each other, reducing the stagnation effect observed in the nonnegative-coefficient case. Therefore, the benefit of SR becomes less pronounced, and both rounding modes exhibit similar accuracy.

In both panels of~\Cref{fig:fp-experiments}, choosing $r$ close to $ \lceil \log_2(4000)/2 \rceil = 6$
already provides good accuracy. Increasing $r$ beyond this value does not lead to significant additional improvements in the relative error. This observation is consistent with the theoretical rule of thumb suggesting that about $\lceil (\log_2 N)/2 \rceil$ random bits are sufficient to obtain the expected probabilistic error behavior. It can also be observed that some \textit{binary16} RN errors are exceptionally small. This behavior can be explained by the fact that RN rounding may alternately underestimate and overestimate the exact intermediate values during the execution of Horner's scheme. In some favorable cases, these rounding errors partially cancel each other, leading to a final computed value that is very close to the exact result.

\begin{figure}
    \centering
    \includegraphics[width=0.5\linewidth]{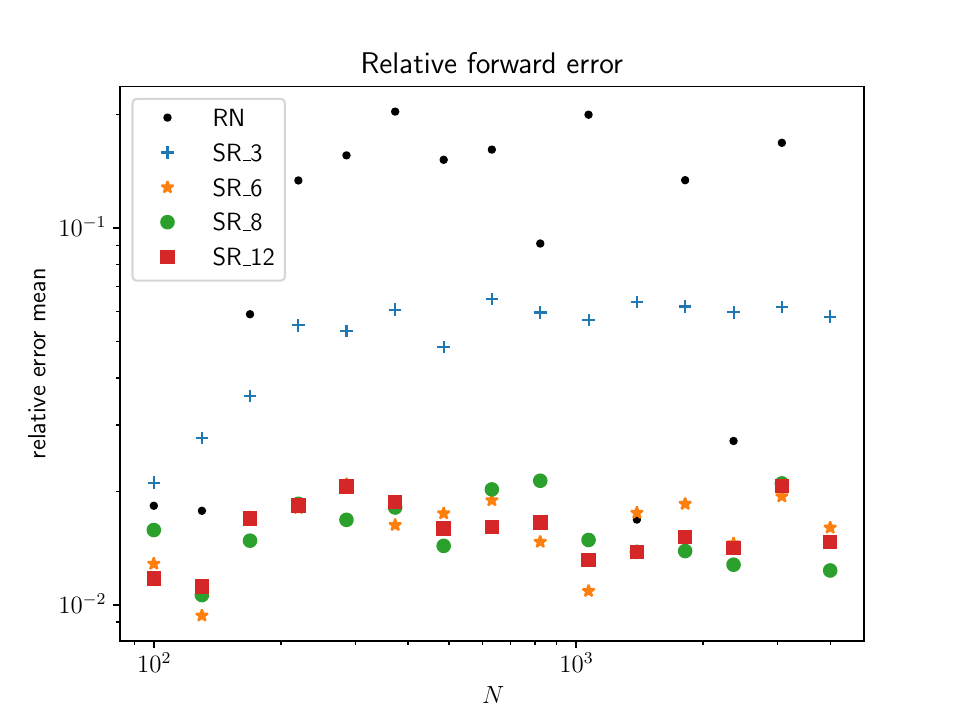}%
    \includegraphics[width=0.5\linewidth]{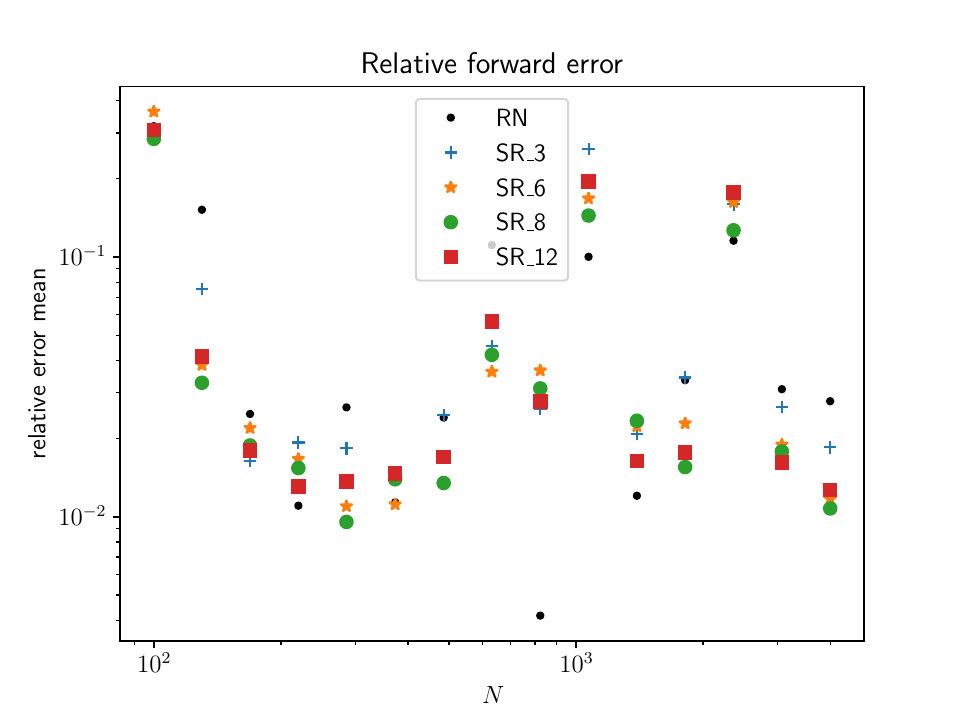}
    \caption{The analogous experiment to~\Cref{fig:fp-experiments} for \textit{bfloat16} arithmetic ($p=8$).}
    \label{fig:bf-experiments}
\end{figure}

\Cref{fig:bf-experiments} illustrates the same qualitative behavior as \Cref{fig:fp-experiments}. In particular, for coefficients in $[0,1]$, the stagnation effect observed with RN is again mitigated when using SR, and the experiments confirm the same guideline for the choice of the number of random bits $r$.
Interestingly, for \textit{bfloat16}, even with a small number of random bits ($r=3$), SR yields smaller errors than RN when the coefficients are drawn from $[-1,1]$.
This behavior can be related to the characteristics of the \textit{bfloat16} format that has a largeer exponent range and a smaller mantissa than \textit{float16}. As a consequence, the spacing between consecutive floating-point numbers is relatively large, so rounding errors are coarse and can accumulate more easily with deterministic rounding. Introducing SR, even with a small number of random bits, helps decorrelate these rounding errors and reduces the accumulation effects observed with RN.

\begin{remark}
  Both figures are evaluated using the largest values $x$ closest to 1 in the target formats, which is challenging for RN in low precision arithmetic and leads to stagnation. For more moderate values, such as $x = 0.5$, the behavior is different. The powers of $x$ decrease rapidly, which delays the occurrence of stagnation. As a result, RN typically provides accurate results, and the advantage of SR is less pronounced.
\end{remark}

\subsection{Pairwise summation}

We now look at pairwise summation in \textit{bfloat16} arithmetic. The goal is to assess the behavior of $\SR_{p,r}$ with a limited number of random bits $r$, and to compare its accuracy with RN. In contrast to the previous experiments, we plot only one run of $\SR_{p,r}$.

\begin{figure}
    \centering
    \includegraphics[width=0.5\linewidth]{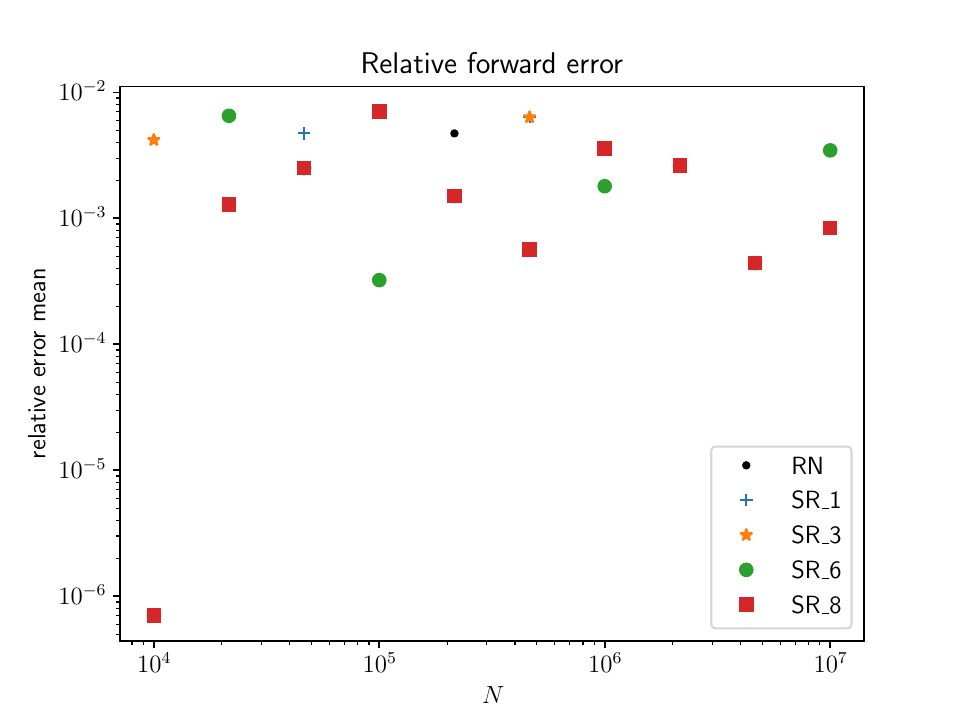}%
    \includegraphics[width=0.5\linewidth]{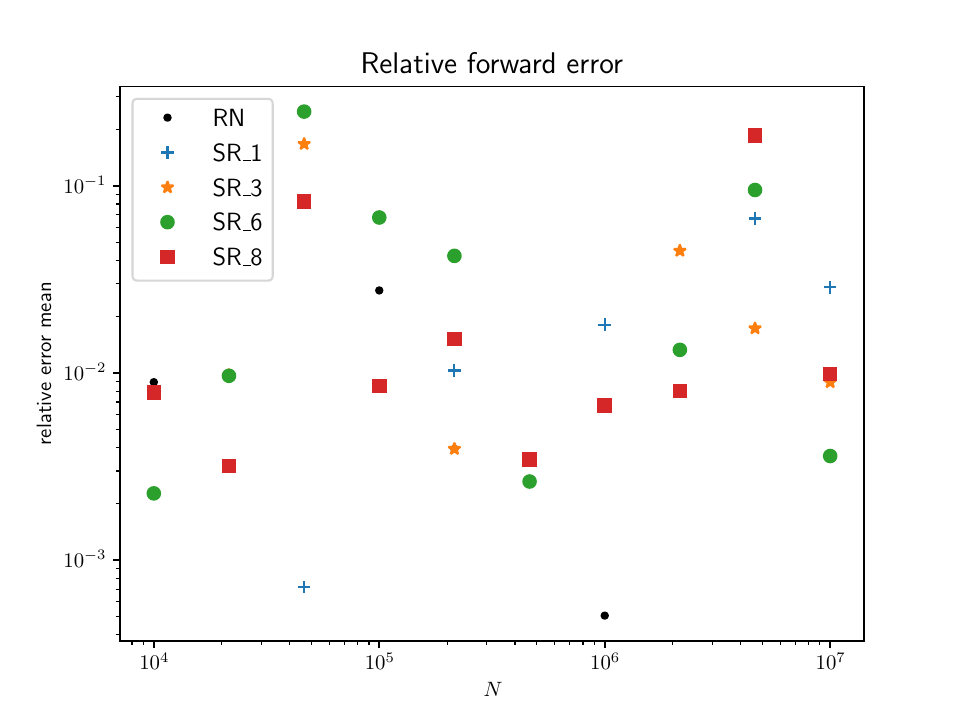}
    \caption{Relative error of RN and $\SR_{8,r}$ in IEEE-754 \textit{bfloat16} arithmetic ($p=8$) for pairwise summation. The floating point values are drawn from a uniform distribution over $[0,10^5]$ (left) and $[-10^5,10^5]$ (right).}
    \label{fig:bf-pairwise}
\end{figure}

The results presented in \Cref{fig:bf-pairwise} show that, the relative errors obtained with SR and RN are of comparable magnitude in both settings. This indicates that, unlike in more ill-conditioned accumulation patterns, the pairwise summation algorithm avoids stagnation, which reduces the potential benefit of SR.
Moreover, the experiments are in agreement with the theoretical analysis. In particular, choosing $r \approx \lceil \log_2(\log_2(10^7))/2 \rceil \approx 3$ is sufficient to achieve good accuracy. Increasing $r$ beyond this value does not lead to noticeable improvements.

\section{Conclusion}
\label{sec:conclusion}

We investigated the impact of limited-precision SR on two fundamental numerical algorithms: polynomial evaluation using Horner's algorithm and pairwise summation. By explicitly accounting for the number of random bits $r$, we derived probabilistic error bounds under $\SR_{p,r}$ arithmetic. The bounds are proportional to $\sqrt{n} u_p + n u_{p+r}$ for the Horner algorithm and to $\sqrt{\log_2(n)} u_p + \log_2(n) u_{p+r}$ for pairwise summation. Our results confirm the same model of probabilistic bound proved before for recursive summation and inner product under $\SR_{p,r}$.
In particular, to choose an appropriate $r$, the rule of thumb for algorithms dominated by computation chains with error propagation of length at most $k$ is to consider $r\approx \lceil\log_2(k)/2 \rceil$. Increasing $r$ beyond this threshold yields only marginal improvements.
The numerical experiments corroborate the theoretical findings, in particular, the rule of thumb remains valid for both situations, with and without stagnation.

\bibliographystyle{spmpsci}
\bibliography{references}

\end{document}